\documentclass[11pt]{article}

 \usepackage{epstopdf}
\usepackage{amsmath}
\usepackage{amsfonts,bm}
\usepackage{color}
\usepackage{amssymb}
\usepackage{graphicx}
\usepackage{enumerate}
\usepackage{amsthm,amscd}
\usepackage[all]{xy}

 \allowdisplaybreaks

\usepackage{hyperref}

\usepackage{mathrsfs,amscd,amssymb,amsthm,amsmath,bm,graphicx,psfrag,subfigure,url}

\newtheorem{theorem}{Theorem}[section]
\newtheorem{corollary}[theorem]{Corollary}

\newtheorem{question}[theorem]{Question}
\newtheorem{lemma}[theorem]{Lemma}
\newtheorem{claim}{Claim}[section]

\begin{document}

\title{Sufficient spectral conditions 
for graphs being $k$-edge-Hamiltonian   
or $k$-Hamiltonian\thanks{This paper was firstly announced in September 2021, 
and was later published on Linear and Multilinear Algebra, 2022. 
This is a final version;   
see \url{https://doi.org/10.1080/03081087.2022.2093321}. 
E-mail addresses: ytli0921@hnu.edu.cn (Y. Li), 
ypeng1@hnu.edu.cn (Y. Peng, corresponding author).}}

\author{Yongtao Li,  Yuejian Peng$^{*}$  \\  
{\small School of Mathematics, Hunan University} \\
{\small Changsha, Hunan, 410082, P.R. China }  \\ 
}

\maketitle

\vspace{-0.5cm}

\begin{abstract}
A graph $G$ is $k$-edge-Hamiltonian 
if  any collection of vertex-disjoint paths with at most $k$ edges altogether belong to a Hamiltonian cycle in $G$. 
A graph $G$ is $k$-Hamiltonian if for all $S\subseteq V(G)$ 
with $|S|\le k$, the subgraph induced by $V(G)\setminus S$ 
has a Hamiltonian cycle. These two concepts are classical extensions 
for the usual Hamiltonian graphs. 
In this paper, we present some spectral sufficient 
conditions for a graph to be $k$-edge-Hamiltonian 
and $k$-Hamiltonian in terms of the adjacency 
spectral radius as well as the signless Laplacian spectral radius. 
Our results could be viewed as slight extensions of the recent theorems proved  
by Li and Ning [Linear Multilinear Algebra 64 (2016)], 
Nikiforov [Czechoslovak Math. J. 66  (2016)]  
and  Li, Liu and Peng [Linear Multilinear Algebra 66 (2018)]. 
Moreover, we shall prove a stability result for  
graphs being $k$-Hamiltonian, which could be regarded as 
a complement of two recent results of F\"{u}redi, Kostochka and Luo 
[Discrete Math. 340 (2017)] and [Discrete Math. 342 (2019)].
 \end{abstract}



\section{Introduction}

Let $G=(V,E)$ be a simple graph with vertex set $V$ 
and edge set $E$. 
The order of $G$ is defined by $|V|$ and the size by $|E|$. 
We usually write $m$ and $n$ for the size and the order of $G$ respectively. 
For disjoint subsets $A,B\subseteq V$, we let $e(A,B)$ denote the number of edges of $G$ with one end-vertex in $A$ and the other in $B$. Let $d_G(v)$ (or $d(v)$ if there is no confusion)
be the degree of a vertex $v$ in $G$, and  let $\delta(G)$ be the minimum degree of $G$. 
We write $K_s$ for the complete graph on $s$ vertices, 
and $I_t$ for the independent set with $t$ vertices. 
Let $\omega (G)$ be the number of vertices of a largest 
complete subgraph in $G$. 
  For two vertex-disjoint graphs  $G$ and $H$, we use $G\cup H$ 
   to denote the disjoint union  of $G$ and $H$,  
 and  we write $G\vee H$ for the join graph of $G$ and $H$, 
 which is a graph obtained from $G\cup H$ by adding 
all edges between $G$ and $H$.  

The \emph{adjacency matrix} of $G$ is $A(G)=(a_{ij})_{n \times n}$, whose entries satisfy $a_{ij}=1$ if two vertices $i$ and $j$ are adjacent in $G$, and $a_{ij}=0$ otherwise. The \emph{characteristic polynomial} of $G$ is $P_G(x)=\det(xI -A(G))$, and the \emph{eigenvalues} of $G$ are the roots of $P_G(x)$ (with multiplicities). Clearly $A(G)$ is a real symmetric   matrix, so the eigenvalues of $G$ are real. The largest eigenvalue of $G$ is called the {\it spectral radius} of $G$ and is denoted by $\lambda (G)$.

Let $d_i$ be the degree of vertex $v_i$ and 
$D(G)$ be the  diagonal  matrix of degrees, that is, 
$D(G)=\mathrm{diag}(d_1,d_2,\ldots ,d_n)$. 
The {\it signless Laplacian matrix} of $G$ is defined as 
$Q(G)=D(G) + A(G)$, which is also a real symmetric   matrix, 
so the eigenvalues of $Q(G)$ are  real numbers. 
The eigenvalues of $Q(G)$ are said to be the 
signless Laplacian eigenvalues of $G$. 
The largest eigenvalue of $Q(G)$ is called {\it the 
signless Laplacian spectral radius} of $G$, and denoted by 
$q(G)$.

In the study of spectral graph theory, 
there are various matrices that are associated with a graph, 
such as the adjacency matrix,  the Laplacian matrix, 
signless Laplacian matrix and distance matrix. 
One of the main problems of algebraic graph theory 
is to determine the combinatorial properties of graphs that are reflected from 
the  algebraic properties of such matrices; 
see \cite{Bap2014,BH2012,GR2001} for more details. 
In this paper, we mainly focus on the 
adjacency spectral radius and signless Laplacian spectral radius.

\subsection{Hamiltonicity of graphs} 

A cycle passing through all  vertices of a graph is called a  Hamilton cycle. A graph containing a Hamilton cycle is called a Hamiltonian graph. A path passing through all  vertices of a graph is called a Hamiltonian path and a graph containing a Hamiltonian path is said to be traceable.  

Every complete graph on at least three vertices is evidently Hamiltonian, as the vertices of a Hamilton cycle can be selected one by one in an arbitrary order. 
Conditions to guarantee the existence of a Hamilton cycle have 
been studied actively. 
In particular, we may
ask how large the minimum degree can be in order to guarantee the existence of
a Hamilton cycle. The  celebrated  Dirac theorem 
\cite{Dirac52} answered this question. 
It states that every graph with $n\ge 3$ vertices and minimum degree at least $\frac{n}{2}$ has
a Hamilton cycle. 
The condition is sharp when we consider the complete bipartite 
graph with the parts of sizes $\lfloor \frac{n-1}{2}\rfloor$ and 
$\lfloor  \frac{n+1}{2}\rfloor$.

The following result is due to 
Ore \cite{ore60} and Bondy \cite{Bondy72} independently. 
It is a direct consequence of  Chv\'{a}tal's theorem 
on degree sequences; 
see \cite[p. 60]{Bondy76} for more details. 

\begin{theorem}[Ore \cite{ore60}, Bondy \cite{Bondy72}] \label{coroore}
Let $G$ be a graph on $n\ge 3$ vertices. If 
\[ e(G)\ge {n-1 \choose 2} +1, \]
then  $G$ has a Hamilton cycle or 
$G=K_1\vee (K_1\cup K_{n-2})$ or $n=5$ and $G=K_2 \vee I_3$.  
\end{theorem}

In 1962, Erd\H{o}s improved the above result 
for graphs with given minimum degree. 

\begin{theorem}[Erd\H{o}s \cite{erdos62}] \label{thmerdos62}
Let $G$ be a graph of order $n$. 
If   the minimum degree $\delta (G)\ge \delta$ 
where $1\le \delta \le \frac{n-1}{2}$ and     
\begin{equation} \label{eq1}
 e(G)> \max\left\{ {n- \delta \choose 2}+\delta^2, 
{n- \lfloor \!\frac{n-1}{2} \! \rfloor \choose 2}+
\left\lfloor \! \frac{n-1}{2} \! \right\rfloor^2 \right\},  
\end{equation}
then $G$ has a Hamilton cycle. 
\end{theorem}

We remark here that 
the condition $\delta \le \frac{n-1}{2}$ is reasonable 
since if $\delta > \frac{n-1}{2}$ then  $ \delta \ge \frac{n}{2}$, 
the well-known Dirac theorem guarantees that  $G$ must be Hamiltonian. 
To see the sharpness of the bound in Theorem \ref{thmerdos62}, 
we consider the graph 
$H_{n,\delta}$ obtained from a copy of $K_{n- \delta} $ 
by adding an independent set of $\delta$ vertices with degree $\delta$ 
each of which is adjacent to the same $\delta$ vertices in $K_{n- \delta}$. 
In the language of graph join and union, 
that is, 
\begin{equation} \label{defh}
  H_{n,\delta}:=K_{\delta} \vee ( K_{n-2\delta}\cup I_{\delta}).
  \end{equation}
Clearly, $H_{n,\delta}$ does not contain a Hamilton cycle 
and  $e(H_{n, \delta})={n- \delta \choose 2} + \delta^2$.  
When $n\ge 6 \delta$, we can see that $e(H_{n,\delta})$ 
attains the maximum in (\ref{eq1}).  
Thus, we can get the following corollary. 

\begin{corollary}[Erd\H{o}s] \label{coro34}
Let $\delta \ge 1$  and $n\ge 6 \delta$. 
If $G$ is an $n$-vertex graph with $\delta (G)\ge \delta$ 
and 
$e (G)\ge e(H_{n,\delta})$, 
then either $G$ has a Hamilton cycle
or $G=H_{n,\delta}$. 
\end{corollary}

\subsection{Spectral conditions for Hamiltonicity}

In 2010, Fiedler and Nikiforov 
proved a spectral version of  Theorem \ref{coroore}. 

\begin{theorem}[Fiedler--Nikiforov \cite{FiedlerNikif}]
If $G$ is a graph on $n\ge 3$ vertices and 
\begin{eqnarray*}  
\lambda (G)> n-2,
\end{eqnarray*}
then either $G$ has a Hamilton cycle 
or $G=K_1\vee (K_1\cup K_{n-2})$. 
\end{theorem}

This result motivated a large number of researches  
on the topic of finding a spectral condition to 
guarantee the existence of a Hamilton cycle 
and path; see, e.g., \cite{YF13, Ningbo15,Ningbo20,Liuruifang,FengLAA17,LiuDMGT,LiBinlong,LNLAA17,Niki16}.
In 2013, Yu and Fan \cite{YF13} 
gave the corresponding version 
for the signless Laplacian spectral radius. 
Recall that $q(G)$ stands for the 
  signless Laplacian spectral radius, i.e., 
 the largest eigenvalue of 
 the {\it signless Laplacian matrix} $Q(G)=D(G) + 
 A(G)$, where $D(G)=\mathrm{diag}(d_1,\ldots ,d_n )$ 
 is the degree diagonal matrix and 
 $A(G)$ is the adjacency matrix. 

\begin{theorem}[Yu--Fan \cite{YF13}]
If $G$ is a graph on $n\ge 3$ vertices and 
\begin{eqnarray*}  
q(G)>2(n-2),
\end{eqnarray*}
then  $G$ has a Hamilton cycle  
or $G=K_1\vee (K_1\cup K_{n-2})$, 
or $n=5$ and $G=K_2 \vee I_3$. 
\end{theorem} 

In \cite{YF13}, 
the counterexample of 
$n= 5,G=K_2\vee I_3$ is missed.  
This tiny flaw has already been pointed 
out  by Liu, Shiu and Xue \cite{Liuruifang} and 
by Li and Ning \cite{LiBinlong} as well.

By introducing the minimum degree of a graph as a new parameter, Li and Ning \cite{LiBinlong} extended
Fiedler and Nikiforov's results \cite{FiedlerNikif} 
 in some sense and obtained a spectral analogue 
 of Theorem \ref{thmerdos62}  of Erd\H{o}s 
 on the existence of Hamilton cycles. 

\begin{theorem}[Li--Ning \cite{LiBinlong}] \label{thmln16a}
Suppose $\delta \ge 1$  and $n\ge \max\{6 \delta+5,(\delta^2+6\delta+4)/2\}$. 
If $G$ is an $n$-vertex graph with $\delta (G)\ge \delta$ and  
\begin{eqnarray*}  
\lambda (G)\ge \lambda (H_{n,\delta}),
\end{eqnarray*}
then either $G$ has a Hamilton cycle or $G=H_{n,\delta}$. 
\end{theorem}

\begin{theorem}[Li--Ning \cite{LiBinlong}] \label{thmln16b}
Suppose $\delta \ge 1$ and $n\ge \max\{6\delta+5,(3\delta^2+5\delta+4)/2\}$. 
If $G$ is an $n$-vertex graph with $\delta (G)\ge \delta$ and 
\begin{eqnarray*}  
q (G)\ge q (H_{n,\delta}),
\end{eqnarray*}
then either 
$G$ has a Hamilton cycle 
or $G=H_{n,\delta}$. 
\end{theorem}

Although these results of Li and Ning seem to be algebraic, 
their proof of  Theorem \ref{thmln16a} and Theorem \ref{thmln16b} 
 also need detailed 
graph structural analysis. 
The key ingredients of the proof of these theorems  are mainly based on a  stability result  \cite[Lemma 2]{LiBinlong}. 
We denote  
\begin{equation}   \label{defl}
 L_{n,\delta}:=K_1 \vee (K_\delta \cup K_{n-\delta-1}).
 \end{equation} 
Clearly $L_{n,\delta}$ contains no Hamilton cycle and 
$e(L_{n,\delta}) ={n-\delta \choose 2} + {\delta +1 \choose 2} 
< e(H_{n,\delta})$.

\medskip 
Soon after, 
Nikiforov \cite[Theorem 1.4]{Niki16} proved the following theorem.

\begin{theorem}[Nikiforov \cite{Niki16}] \label{thmniki}
Suppose that $\delta\ge 1$ and $n\ge \delta^3+\delta+4$. 
If $G$ is an $n$-vertex graph with  minimum degree 
$\delta (G)\ge \delta$ and  
\begin{eqnarray*}
\lambda (G) \ge n- \delta -1,
\end{eqnarray*}
then  $G$ has a Hamilton cycle, or $G= H_{n,\delta}$,   
or $G= L_{n,\delta}$.  
\end{theorem}

We now introduce an important operation of graphs, 
which is known as the celebrated Kelmans operation; 
see, e.g., \cite{Kel1981} or \cite[p. 36]{BH2012}. 
Let $G$ be a graph and $u,v\in V(G)$ be distinct vertices. 
We define a new graph $G^*$ obtained from $G$ by replacing the edge 
$\{v,x\}$ by a new edge $\{u,x\}$ for all $x\in N(v) \setminus (N(u)\cup \{u\})$, 
and all vertices different from $u, v$ remain unchanged. 
Note that vertices $u,v$ are adjacent in $G^*$ if and only if 
they are adjacent in $G$. 
An isomorphic graph is obtained if the roles of $u$ and $v$ are interchanged. 
Furthermore, Csikv\'{a}ri \cite{Csi2009} showed that the Kelmans operation 
does not decrease the spectral radius  of a graph. 
Correspondingly, the same result also holds for the signless Laplacian spectral radius 
as well \cite[Theorem 2.12]{LiBinlong}. Additionally, an analogous variant of the Kelmans operation can be seen in \cite{WXH2005}.

Theorem \ref{thmniki} is a slight improvement on 
Theorem \ref{thmln16a}. Indeed, we observe that 
$K_{n-\delta}$ is a proper subgraph of $H_{n,\delta}$ and $L_{n,\delta}$, 
which yields 
  $\lambda (H_{n,\delta}) > \lambda (K_{n-\delta})=n-\delta-1$ and 
  $\lambda (L_{n,\delta}) > \lambda (K_{n-\delta})=n-\delta-1$. 
 Moreover, 
 by  applying the Kelmans operations on $L_{n,\delta}$, 
 we can obtain a proper subgraph of $H_{n,\delta}$.  
 Hence we can get  $\lambda (H_{n,\delta}) > \lambda (L_{n,\delta})$. 
 By calculation, we know that 
 $\lambda (H_{n,\delta})$ is very close to $n-\delta -1$ with $n$ sufficiently large.

Although Nikiforov's Theorem \ref{thmniki} 
 strengthened slightly Theorem \ref{thmln16a}.   
Unfortunately, one dissatisfaction in Theorem \ref{thmniki} 
is that 
the requirement of the order of graph is stricter than 
that in Theorem \ref{thmln16a}.  
One open problem is to relax this requirement.  
In addition, a natural question is that whether 
the value bound 
$q(G)\ge 2(n-\delta -1) $ corresponding to Theorems \ref{thmln16b}
 hold or not. The similar problems under the conditions of signless Laplacian spectral
radius of graph seems much more complicated 
since we can delete some edges from the clique $K_{n-\delta}$ and still keep the signless Laplacian spectral radius no less than $2(n-\delta -1)$.

\medskip 

In 2018, Li, Liu and Peng \cite{LLPLMA18} 
 solved this problem completely. 
They 
gave the corresponding improvement on 
the result of the signless Laplacian spectral radius 
in Theorem \ref{thmln16b}. 
Interestingly, the extremal graphs in their theorems are quite different from those in
Theorems \ref{thmln16b}.  
Recalling the definition in (\ref{defh}) and (\ref{defl}),  
we denote $X=
\{v\in V(H_{n,\delta}) : d(v)=\delta\}, 
Y= \{v\in V(H_{n,\delta}): d(v)=n-1\}$ 
and $Z=\{v\in V(H_{n,\delta}): d(v)=n- \delta -1\}$. 
Let $E_1(H_{n,\delta})$ be the set of those edges of $H_{n,\delta}$ 
whose both endpoints are from $Y\cup Z$. 
We define 
\[ \mathcal{H}_{n,\delta}^{(1)} = 
\left\{ H_{n,\delta} \setminus E': E' \subseteq E_1(H_{n,\delta})
 ~\text{with}~|E'|\le 
\lfloor {\delta^2}/{4}\rfloor \right\}. \]
Similarly, for the graph $L_{n,\delta}$, 
we denote 
$X=\{v\in V(L_{n,\delta}): d(v)=\delta \}, 
Y=\{v\in V(L_{n,\delta}): d(v)=n-1\}$ 
and $Z=\{v\in V(L_{n,\delta}): d(v)=n-\delta -1\}$. 
The notation is clear although we used the same alphabets 
to denote the sets of vertices. 
It is easy to see that $Y$ contains only one vertex. 
We use $E_1(L_{n,\delta})$ to denote 
the set of edges of 
$L_{n,\delta}$ whose both endpoints are from $Y\cup Z$. 
We define 
\[ \mathcal{L}_{n,\delta}^{(1)} = 
\left\{ L_{n,\delta} \setminus E': E' \subseteq E_1(L_{n,\delta})
 ~\text{with}~|E'|\le 
\lfloor {\delta}/{4}\rfloor \right\}. \]

\begin{theorem}[Li--Liu--Peng \cite{LLPLMA18}] \label{thmllp}
Assume that $\delta \ge 1$ and 
$n\ge \delta^4+\delta^3+4\delta^2 +\delta+6$. 
Let $G$ be a connected graph with $n$ vertices and minimum 
degree $\delta (G)\ge \delta$. If 
\[ q(G)\ge 2(n-\delta -1),  \]
then $G$ has a Hamilton cycle unless 
$G\in \mathcal{H}_{n,\delta}^{(1)}$ 
or $G\in \mathcal{L}_{n,\delta}^{(1)}$. 
\end{theorem}

The paper is organized as follows. In Section \ref{sec2}, 
we shall present our results on the problems involving the 
existence of Hamilton cycles. 
This paper is mainly motivated by the aforementioned works \cite{LiBinlong, Niki16, LLPLMA18}. 
Our theorems extend  
Theorems \ref{thmln16a}--\ref{thmllp} slightly, we shall provide  
the sufficient conditions on graphs being $k$-edge-Hamiltonian 
and $k$-Hamiltonian.  
In Section \ref{sec3}, 
we review some basic preliminaries for our  use. 
Moreover, we shall prove a stability result on 
graphs being $k$-Hamiltonian (Theorem \ref{thm34}). 
This theorem can be regarded as  the complement of 
two recent results showed by F\"{u}redi, Kostochka and Luo \cite{FKL17,FKL19}. 
In Section \ref{sec4}, we shall give the complete 
proofs of our main results stated in Section \ref{sec2}.

\section{Main results}

\label{sec2}

\subsection{Spectral conditions for $k$-edge-Hamiltonicity}   

A graph $G$ is called $k$-edge-Hamiltonian 
if any collection of at most $k$ edges consisting of vertex-disjoint paths is contained in a Hamilton cycle in $G$. 
In other words, each linear forest  with at most $k$ 
edges in $G$ can be extended to a Hamilton cycle of $G$. 
In particular, being $0$-edge-Hamiltonian is 
equivalent to being Hamiltonian. 

In this section, we shall present our theorems 
on the sufficient spectral conditions 
for graphs being $k$-edge-Hamiltonian. 
For convenience, we denote 
\begin{equation}  \label{eqhnkd} 
  H_{n,k,\delta } := K_\delta \vee ( K_{n-2\delta +k} \cup I_{\delta -k})   
 \end{equation} 
 and 
 \begin{equation}  
  L_{n,k,\delta}:=K_{k+1}\vee( K_{n-\delta-1}\cup K_{\delta-k}). 
  \end{equation}
It is easy to see that 
both $H_{n,k,\delta } $ and $L_{n,k,\delta}$ 
 have minimum degree 
$\delta (H_{n,k,\delta })=\delta (L_{n,k,\delta})=\delta$. 
Moreover, $H_{n,k,\delta}$ is not $k$-edge-Hamiltonian since 
no linear forest with $k$ edges within the dominating 
clique $K_{\delta}$ can be contained in a Hamilton cycle. 
Similarly,  $L_{n,k,\delta}$ is not $k$-edge-Hamiltonian as 
no path with $k$ edges within the dominating 
clique $K_{k+1}$ can be extended to a Hamilton cycle. 
 In particular, by setting $k=0$, we can see that 
 $H_{n,0,\delta}$ is  the same as $H_{n,\delta}$, 
 and $L_{n,0,\delta}$ is the same as $L_{n,\delta}$, which are 
 defined in equations (\ref{defh}) and (\ref{defl}).

\begin{gather*} 
\includegraphics[scale=0.16]{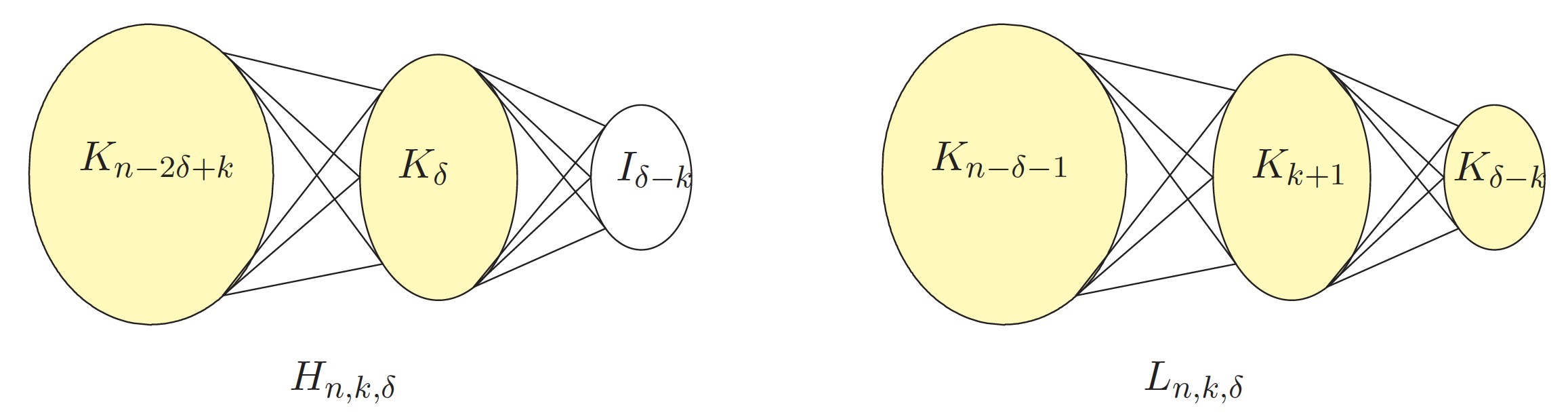}  
\end{gather*}

Now, we are ready to present our results in this paper. 
To avoid unnecessary calculations, 
we do not attempt to get the best bound 
on the order of graphs in the proof.

 \begin{theorem}   \label{thm21}
Let $k \geq 0$, $\delta\geq k+2$ and $n$ be sufficiently large.  
If $G$ is an $n$-vertex graph  with minimum degree $\delta(G)\geq \delta$ and 
  $$\lambda(G)\geq n-\delta+k-1,$$
then $G$ is $k$-edge-Hamiltonian unless $G=H_{n,k,\delta }$ 
or $G=L_{n,k,\delta}$.
\end{theorem}

Since  $H_{n,k,\delta }$ contains $K_{n-\delta+k}$ as a proper subgraph, we have 
$\lambda(H_{n,k,\delta } )> n-\delta+k-1$. 
Moreover, applying the Kelmans operations on 
$L_{n,k,\delta}$, we can get a proper subgraph of $H_{n,k,\delta}$, 
this implies 
 $\lambda (H_{n,k,\delta }) > \lambda (L_{n,k,\delta})$; 
 see, e.g., \cite[Theorem 2.12]{LiBinlong}. 
 With this observation in mind, Theorem \ref{thm21} 
 implies the following corollary, which is an extension on Theorem \ref{thmln16a}.

 \begin{corollary} 
Let $k \geq 0$, $\delta\geq k+2$ and $n$ be sufficiently large. 
If $G$ is an $n$-vertex graph  with minimum degree $\delta(G)\geq \delta$ and 
  $$\lambda(G)\geq\lambda(H_{n,k,\delta } ) ,$$
then $G$ is $k$-edge-Hamiltonian unless $G=H_{n,k,\delta }$. 
\end{corollary}

Recall that 
$H_{n,k,\delta } = K_\delta \vee ( K_{n-2\delta +k} \cup I_{\delta -k})$.  
Let $X$ be the set of $\delta -k$ vertices with degree $\delta$ forming 
by the independent set $I_{\delta -k}$, 
$Y$ be the set of $\delta$ vertices  with degree $n-1$ corresponding to 
 the clique $K_{\delta}$ 
and $Z$ be the set of the remaining 
$n-2\delta +k$ vertices with degree $n-\delta +k-1$ 
corresponding to the clique 
$K_{n-2\delta +k}$. 
We write $E_1(H_{n,k,\delta })$ for the 
set of edges of $H_{n,k,\delta }$ 
whose both endpoints are from $Y\cup Z$. 
Furthermore, we define 
the family $\mathcal{H}_{n,k,\delta}^{(1)}$ of graphs 
as below. 
\[ \mathcal{H}_{n,k,\delta}^{(1)} = 
\left\{ H_{n,k,\delta} \setminus E': E' \subseteq E_1(H_{n,k,\delta})
 ~\text{with}~|E'|\le 
\lfloor {\delta (\delta -k)}/{4}\rfloor \right\}. \]
Here, we write $H_{n,k,\delta} \setminus E'$ for the graph obtained from 
$H_{n,k,\delta}$ by deleting all edges of $E'$. 
Similarly, for the graph $L_{n,k,\delta} 
=K_{k+1}\vee( K_{n-\delta-1}\cup K_{\delta-k})$,  
we denote  
$X$ by the set of vertices with degree $\delta$ corresponding to 
the clique $K_{\delta -k}$, 
$Y$ by the set of vertices with degree $n-1$
corresponding to the clique $K_{k+1}$, 
and $Z$ by the set of the 
remaining $n-\delta -1$ vertices with degree $n-\delta +k-1$. 
We write $E_1(L_{n,k,\delta })$ for the 
set of edges of $L_{n,k,\delta }$ 
whose both endpoints are from $Y\cup Z$. 
Moreover, we define 
the family $\mathcal{L}_{n,k,\delta}^{(1)}$ of graphs 
as follows. 
\[ \mathcal{L}_{n,k,\delta}^{(1)} = 
\left\{ L_{n,k,\delta} \setminus E': E' \subseteq E_1(L_{n,k,\delta})
 ~\text{with}~|E'|\le 
\lfloor {(k+1) (\delta -k)}/{4}\rfloor \right\}. \]

In this paper, we also 
present 
the following sufficient 
 conditions on the signless Laplacian spectral radius 
 for $k$-edge-Hamiltonian graphs with large minimum 
 degree.

\begin{theorem} \label{thm23}
Let $k\ge 0,\delta \ge k+2$ and 
$n$ be sufficiently large. 
If $G$ is an $n$-vertex graph  with minimum degree $\delta(G)\geq \delta$ and 
  $$q (G)\geq 2(n-\delta+k-1),$$
then $G$ is $k$-edge-Hamiltonian 
unless $G\in \mathcal{H}_{n,k,\delta }^{(1)}$ 
or $G\in \mathcal{L}_{n,k,\delta}^{(1)}$.
\end{theorem}

As a consequence, we get the following corollary. 

 \begin{corollary} 
Let $k \geq 0$, $\delta\geq k+2$ and $n$ be sufficiently large. 
If $G$ is an $n$-vertex graph with minimum degree $\delta(G)\geq \delta$ and 
  $$q(G)\geq  q(H_{n,k,\delta } ) ,$$
then $G$ is $k$-edge-Hamiltonian unless $G=H_{n,k,\delta }$. 
\end{corollary}

\subsection{Spectral conditions for $k$-Hamiltonicity}

A graph $G=(V,E)$ is called  \emph{$k$-Hamiltonian} if for all $X\subseteq V$ with $|X|\leq k$, the subgraph induced by the set 
$V\setminus X$ is Hamiltonian. In particular,  $0$-Hamiltonian 
graph is the same as the general  Hamiltonian graph.  
 In \cite{Chartrand70,Chvatal72}, 
 it is obtained that for a graph $G$,
 if $\delta(G)\geq \frac{n+k}2,$
  then $G$ is $k$-Hamiltonian. 
  Clearly, when $k=0$, it reduces to the Dirac theorem. 

 Recently, by utilizing   the degree sequences and the closure concept,  
  Liu, Liu, Zhang and Feng  \cite{LiuDMGT} generalized Theorem \ref{thmniki} to 
   $k$-Hamiltonian graphs. 
   Moreover, 
   Liu, Lai and Das \cite{LLD2019} proved some further results on $k$-Hamiltonian graphs independently. 
   The theorems in our paper  could be viewed as slight improvements 
   on partial results of  \cite{LLD2019}, 
   since the conditions in our theorems are more concise and 
   the extremal graphs seems more accurate.  
We mention here that 
there is a tiny typo at the end of the proof in \cite[Theorem 4]{LiuDMGT}  
since the extremal graph is not the only one.  
Clearly, the graph $H_{n,k,\delta } 
= K_\delta \vee ( K_{n-2\delta +k} \cup I_{\delta -k})$ is not 
$k$-Hamiltonian and 
$\lambda (H_{n,k,\delta }) > n-\delta +k -1$. 
By a careful modification in \cite{{LiuDMGT}}, 
the correct result should be the following theorem.

 \begin{theorem}  \label{thm25}
Let $k \geq 0$, $\delta\geq k+2$ and $n$ be sufficiently large. 
If $G$ is an $n$-vertex graph  with minimum degree $\delta(G)\geq \delta$ and 
  $$\lambda(G)\geq n-\delta+k-1,$$
then $G$ is $k$-Hamiltonian unless $G=H_{n,k,\delta }$ 
or $G=L_{n,k,\delta}$.
\end{theorem}

In this paper, we shall provide another different way of the proof 
of Theorem \ref{thm25} by applying a stability result 
on the number of edges (see Theorem \ref{thm34}). 
Since  $H_{n,k,\delta }$ contains $K_{n-\delta+k}$ as a proper subgraph, 
it follows that $\lambda(H_{n,k,\delta } )> n-\delta+k-1$. 
On the other hand, by applying the Kelmans operation many times on $L_{n,k,\delta}$, 
we can get a proper subgraph of  $H_{n,k,\delta }$, which leads to 
$\lambda (H_{n,k,\delta }) > \lambda (L_{n,k,\delta})$. 
So we can immediately 
get the following corollary, which extended Theorem \ref{thmln16a} slightly.

 \begin{corollary}[Liu et al. \cite{{LiuDMGT}}]
Let $k \geq 0$, $\delta\geq k+2$ and $n$ be sufficiently large. 
If $G$ is an $n$-vertex graph  with minimum degree $\delta(G)\geq \delta$ and 
  $$\lambda(G)\geq\lambda(H_{n,k,\delta } ) ,$$
then $G$ is $k$-Hamiltonian unless $G=H_{n,k,\delta }$. 
\end{corollary}

In addition, we shall prove the following signless Laplacian spectral 
version. 

\begin{theorem} \label{thm27}
Let $k\ge 0,\delta \ge k+2$ and 
$n$ be sufficiently large. 
If $G$ is an $n$-vertex graph  with minimum degree $\delta(G)\geq \delta$ and 
  $$q (G)\geq 2(n-\delta+k-1),$$
then $G$ is $k$-Hamiltonian 
unless $G\in \mathcal{H}_{n,k,\delta }^{(1)}$ 
or $G\in \mathcal{L}_{n,k,\delta}^{(1)}$.
\end{theorem} 

Similarly, we have the following corollary. 

 \begin{corollary} 
Let $k \geq 0$, $\delta\geq k+2$ and $n$ be sufficiently large. 
If $G$ is an $n$-vertex graph with minimum degree $\delta(G)\geq \delta$ and 
  $$q(G)\geq  q(H_{n,k,\delta } ) ,$$
then $G$ is $k$-Hamiltonian unless $G=H_{n,k,\delta }$. 
\end{corollary}

It is worth noting that 
we spent a lot of efforts in characterizing the extremal families 
in terms of the signless Laplacian radius, which 
is one of the main parts in our paper (see Section \ref{sec4}). 
Furthermore, we have proved that our characterization
 of extremal graphs is sharp. 
In the proof of Theorems \ref{thm23} and \ref{thm27},  
we prove that the extremal graphs are contained in 
$H_{n,k,\delta}$ or $L_{n,k,\delta}$. Furthermore, the sharpness 
of our result can be seen from  Lemmas \ref{lem31} and \ref{lem32}.

\section{Preliminaries and stability results}
\label{sec3}

We need to use the following 
bounds on spectral radius. 
The first bound was found by 
Hong, Shu and Fang \cite{HSF01} 
for connected graphs. 
Independently, Nikiforov \cite[Theorem 4.1]{Niki02} published 
a quite different method of this result for all graphs 
(not necessarily connected). 
Moreover, Zhou and Cho \cite{ZC05} 
determined the graphs which attain the upper bound. 

\begin{theorem}[\cite{HSF01,Niki02}] \label{thm31}
Let $G$ be a graph on $n$ vertices with $\delta (G)\ge \delta$. Then 
\[  \lambda (G) \le \frac{1}{2}\Bigl( \delta-1 + 
\sqrt{8e(G) - 4\delta n + (\delta+1)^2} \Bigr).  \]
\end{theorem}

The following theorem gives an upper bound on 
$q(G)$.

\begin{theorem}[Feng--Yu \cite{FengPIMB09}] \label{thmFY}
Let $G$ be a graph on $n$ vertices. Then 
\[   q(G)\le \frac{2e(G)}{n-1} +n-2. \]
\end{theorem} 

We also need the following operation 
for signless Laplacian spectral radius. 

\begin{lemma}[Hong--Zhang \cite{HZ2005}] \label{lemhz}
Let $G$ be a connected graph 
and $q(G)$ be its signless Laplacian spectral 
radius corresponding to the Perron eigenvector $\bm{x}$.  
Suppose that $u,v$ are two vertices of $G$ and 
 $w_1,w_2,\ldots ,w_s$ are distinct vertices in 
$N(v)\setminus (N(u)\cup \{u\})$ where $1\le s\le d(v)$. 
If $x_u \ge x_v$ and 
$G^*$ is the graph obtained from $G$ by deleting 
the edges $vw_i$ and adding the edges $uw_i$ for all 
$1\le i\le s$, then $q(G)<q(G^*)$. 
\end{lemma}

The above operation is  different from the Kelmans operation 
since we need to compare the coordinates of the Perron eigenvector. 
In addition, we remark that 
the same statement is also valid for the adjacency spectral radius; see, e.g., \cite{WXH2005}.

We next present some graph notations. 
The closure operation introduced by Bondy and Chv\'{a}tal \cite{Bondy} 
is a powerful tool for the problems of Hamiltonicity of graphs. 
Let $G$ be a graph of order $n$. The $s$-closure of $G$, 
denoted by $\mathrm{cl}_s(G)$, 
is the graph obtained from $G$ by recursively joining pairs of 
non-adjacent vertices whose degree sum is at least $s$ until no such pair 
remains.  
It is not hard to prove that the $s$-closure 
of $G$ is uniquely determined; see, e.g., \cite{Bondy}.  
Clearly,  $G$ is a subgraph of $\mathrm{cl}_s(G)$ for every $s$, 
and for any two non-adjacent vertices in $\mathrm{cl}_s(G)$, 
the sum of their degrees is less than $s$.

  \begin{theorem} \label{thm33}
Let ${G}'=\mathrm{cl}_{n+k} (G)$  be the $(n+k)$-closure graph of $G$. \\
(1) \cite{Kro1969, Bondy} 
A graph $G$ is $k$-edge-Hamiltonian 
if and only if   
${G}'$ is $k$-edge-Hamiltonian. \\ 
(2) \cite{Chartrand70, Bondy} 
A graph $G$ is $k$-Hamiltonian 
if and only if  
${G}'$ is $k$-Hamiltonian. 
\end{theorem}

\noindent 
{\bf Remark.} 
From the above discussion, we know that if 
$d_G(u)+d_G(v) \ge n+k$ for all distinct vertices $u,v\in V(G)$, 
then  the closure graph $\mathrm{cl}_{n+k} (G)$ is a complete graph, 
hence it is $k$-Hamiltonian and $k$-edge-Hamiltonian, 
 so is $G$ by Theorem \ref{thm33}.

\medskip 
To prove our theorems, we also need 
the following stability result, which is the main theorem 
proved by F\"{u}redi,  Kostonchka and Luo in \cite[Theorem 5]{FKL19} and 
 is also a generalization of 
the stability result on Hamilton cycle 
proved early in \cite[Lemma 2]{LiBinlong} 
and \cite[Theorem 3]{FKL17} independently.

\begin{theorem}[F\"{u}redi et al. \cite{FKL19}] \label{thmFKL}
Let $\delta >k \ge0 $ and $ n\ge 6\delta -5k +5$. 
If $G$ is an $n$-vertex graph with 
minimum degree $\delta (G) \ge \delta$ 
and 
\[  e(G)> e(H_{n,k,\delta +1}), \]
then $G$ is $k$-edge-Hamiltonian unless 
$G\subseteq H_{n,k,\delta}$ or $G\subseteq L_{n,k,\delta}$.  
\end{theorem}

In this section, we shall 
prove the next stability result for $k$-Hamiltonian graphs, 
which is a complement of Theorem \ref{thmFKL} and 
a generalization of the result in \cite{LiBinlong} and \cite{FKL17}. 
Interestingly, the extremal graphs 
in Theorem \ref{thm34} are the same as those in Theorem \ref{thmFKL}. 

\begin{theorem} \label{thm34}
Let $\delta >k \ge0 $ and $ n$ be sufficiently large. 
If $G$ is an $n$-vertex graph with 
minimum degree $\delta (G) \ge \delta$ 
and 
\[  e(G)> e(H_{n,k,\delta +1}), \]
then $G$ is $k$-Hamiltonian unless 
$G\subseteq H_{n,k,\delta}$ or $G\subseteq L_{n,k,\delta}$.  
\end{theorem} 

We  remark here that 
both 
Theorems \ref{thmFKL} and \ref{thm34} 
can be  proved similarly by applying the techniques 
stated in \cite{LiBinlong} or \cite{FKL17,FKL19} 
or a variant of the proof in \cite[Theorem 1.4]{Niki16}. 
We next include a proof using the method in \cite{LiBinlong} 
with slight differences.

\begin{proof}
Let $G'=\mathrm{cl}_{n+k}(G)$ be the $(n+k)$-closure 
of $G$. 
By Theorem \ref{thm33}, we know that if $G'$ is $k$-Hamiltonian,
 then so is $G$. 
Thus, we now assume that $G'$ is not $k$-Hamiltonian. 
By the definition of closure, we know that 
any two distinct vertices in $G'$ with degree sum 
no less than $n+k$ are adjacent.  
Obviously, we have $\delta (G') \ge \delta (G)\ge \delta$ and $e(G')\ge e(G)$.

\begin{claim} \label{claim31} 
$ \omega (G') =n-\delta + k $. 
\end{claim}

 \begin{proof}[Proof of Claim]
 A vertex of $G'$ is called heavy if it has degree at least $\frac{n+k}{2}$. 
 Since every two vertices whose degree sum is
at least $n+k$ are adjacent, any two heavy vertices are adjacent in $G'$. 
 Namely, the set of all heavy vertices forms a clique in $G'$.  
Let $C$ be the set of vertices of a maximal clique of $G'$ containing all heavy vertices.  
We denote $t=|C|$ and $H=G'\setminus C$  the subgraph of $G'$ induced by 
$V(G')\setminus C$.

\medskip 
We can observe the following two facts:

\begin{itemize}

\item 
For every $v\in V(H)$, we have 
\begin{equation}  \label{bound-eq6}
d_{G'}(v) \le \frac{n+k-1}{2}. 
\end{equation}
 Indeed, otherwise, we may assume $d_{G'}(v)>(n+k-1)/2$,
because $d_{G'}(v)$ is a positive integer, then $d_{G'}(v)\ge (n+k-1)/2+1/2=(n+k)/2$,
so $v$ is contained in $C$, a contradiction. 

\item 
Moreover,  for every $v\in V(H)$, we have 
\begin{equation}  \label{bound-eq7}
d_{G'}(v)\le n+k-t. 
\end{equation}
 Otherwise, we assume that $d_{G'}(v)\ge n+k-t+1$. 
 For each $u\in C$, we have $d_{G'}(u)\ge |C|-1=t-1$ since $C$ is a clique. 
 Note that $d_{G'}(u) + d_{G'}(v) \ge n+k$.  Thus $v$ is  adjacent to $u$ for every $u\in C$. 
The maximality of $C$ implies that $v\in C$, a contradiction.

\end{itemize}

In what follows, we will show that $t\ge n-\delta +k$. 

\medskip 

{\bf Case 1.} Suppose first that 
$1\le t \le n/3 -k/3+\delta +4/3$. 
We mention here that this threshold is determined 
in the forthcoming Case 2. Clearly, for every $v\in V(G')$, we have $d_C(v)\le t-1$, which together with (\ref{bound-eq6}) yields 
\[ e(H) +e(V(H),C) = 
\frac{1}{2} \sum\limits_{v\in V(H)} 
(d_{G'}(v) +d_C(v)) 
\le \frac{1}{2}(n-t)\left( \frac{n+k-1}{2} +t-1\right).  \] 
Then we have 
\begin{align*}
e(G')
&=e(G'[C]) +e(H) +e(V(H),C) \\
&\le \tbinom{t }{ 2} +\tfrac{1}{2}(n-t)\left( \tfrac{n+k-1}{2} +t-1\right) =\tfrac{n-k+1}{4}t +\tfrac{n(n+k-3)}{4} \\
&\le \tfrac{n- k+1}{4}\left( \tfrac{n}{3} -\tfrac{k}{3} +\delta +\tfrac{4}{3}\right) +\tfrac{n(n + k-3)}{4}  \\
&=
\tfrac{1}{3}n^2 + 
\tfrac{3\delta +k -4}{12}n 
+\tfrac{\delta}{4} + \tfrac{k^2}{12} - \tfrac{\delta k}{4} - 
\tfrac{5k}{12} + \tfrac{1}{3}\\
&< e(H_{n,k,\delta +1}),
\end{align*}
the last inequality follows since
 $n$ is large enough, which leads to a contradiction.

\medskip

{\bf Case 2.} 
Secondly, suppose that $n/3 -k/3+\delta +4/3 \le 
t \le n-\delta +k -1$. Note that 
  \[  e(H) +e(V(H),C) \le \sum\limits_{v\in V(H)} d_{G'}(v) 
  \le (n-t)(n+k-t),  \]
  where the last inequality follows by using (\ref{bound-eq7}). 
Therefore
\begin{align*}
e(G')&= e(G'[C]) +e(H) +e(V(H),C) \\
&\le \tbinom{t }{ 2} +(n-t)(n+k-t) 
=\tfrac{3}{2}t^2 - (2n +k + \tfrac{1}{2})t +n(n+k) \\
&\le \tfrac{3}{2}\left( n-\delta +k-1\right)^2 - (2n +k+\tfrac{1}{2})(n-\delta +k-1)+n(n+k) \\
&=e(H_{n,k,\delta +1}), 
\end{align*}
where the last inequality holds since 
the quadratic function on variable $t$ attains the maximum at $t=
n-\delta +k -1$. 
This is also a contradiction.

From the above discussion, we know that 
$\omega (G')\ge |C|\ge n-\delta +k$. 
Suppose that  $C'$ is a largest clique in $G'$ 
with $|C'|\ge n-\delta +k +1$. 
We denote by $H'=G'\setminus C'$ the subgraph of $G'$ induced by $V(G')\setminus C'$. 
Since $G' $ is not a clique (otherwise, $G'$ 
is $k$-Hamiltonian), we get that $V(H')$ is not empty.
Note that $d_{G'}(v) \ge \delta (G)\ge \delta $ for every $v \in V(H')$ and $d_{G'}(u) \ge |C'|-1\ge n + k-\delta $
for every $u\in C'$, hence $d_{G'}(v) + d_{G'}(u)\ge n+k$, 
this means that every vertex in $H'$ is adjacent to every vertex of $C'$,
this contradicts the fact that $C'$ is a maximum clique.

From the above case analysis, 
we now obtain that $\omega (G')=n-\delta +k$. In addition, 
the argument also showed that the set $C$ of vertices with degree 
at least $\frac{n+k}{2}$ induces a maximum clique of $G'$ and $|C|=n-\delta + k$. 
\end{proof}

Let $C$ be the set of vertices of a largest clique in $G'$ and $H=G'\setminus C$. 
By Claim \ref{claim31}, we have $|C| = n-\delta +k$ and $|V(H)|=\delta -k$. 
By the definition of $G'$, we can see 
 that every vertex of $H$ has degree exactly $\delta $ in $G'$. 
We say that a vertex in $C$ is a {\it frontier} vertex if it has degree at least $n-\delta +k$ in $G'$,
that is, it has at least one neighbor in $H$. 
We denote by $F=\{u_1,u_2,\ldots ,u_s\}$ 
the set of all frontier vertices in $C$. 
Since $d_{G'}(u_i) \ge n- \delta +k$ and $d_{G'}(v)\ge \delta $ for every $v\in F$, 
we know that every vertex in $H$ is adjacent to every vertex in $H$, and then 
$N(v)\cap C =F$ for every $v\in V(H)$. 
In fact, we have $d_{G'}(u_i) =n-1$ for every $u_i \in F$. 
 Since $d(v)=\delta$ for every $v\in H$, 
we then get $k+1\le s\le \delta$.  
Since $C$ forms a clique on $n-\delta +k$ vertices,
we can choose a path $P$ in $C-F$ with two end-vertices $u_1$ and $u_s$.  

\begin{claim}
We claim that $s=k+1$ or $s=\delta$.
\end{claim}

\begin{proof}[Proof of Claim]
If $k+2\le s\le \delta -1$, 
we will show that  
$G'$ is $k$-Hamiltonian  in this case. 
Let $S\subseteq V(G')$ be any set of vertices with size at most $k$. 
Set $S_1=S\cap (F\cup V(H))$ and $S_2=S\cap (C\setminus F)$. 
Since $|F|=s\ge k+2$, we have $|F\setminus S_1| \ge 2$, 
so we can fix two vertices $u,v\in F\setminus S_1$. 
We  consider the induced subgraph  $G^*=G'[F\cup V(H)]$ 
and will show that there exists a Hamilton path in $G^*$ 
that  connects vertices $u,v$ and 
lies outside $S_1$.  
Note that for any $x,y\in F\cup V(H)$, we have 
\[  d_{G^*}(x) + d_{G^*}(y) \ge 2\delta \ge
 \delta +s +1= |F\cup V(H)| + k +1.  \]
 By noting the remark of Theorem \ref{thm33}, 
we obtain that $G^*$ is $(k+1)$-Hamiltonian. 
Since $|S_1\cup \{u\}| \le k+1$, 
there exists a Hamilton cycle in the induced subgraph 
$G^*\setminus (S_1\cup \{u\})$, say  
$vv_1v_2\cdots v_r v$. 
Note that $u\in F$ is adjacent to every vertex in $G^*$. In particular, 
$\{u,v_1\}\in E(G^*)$. 
Therefore the path $P_1=uv_1v_2\cdots v_rv$ 
passes through  all vertices 
in $(F\cup V(H))\setminus S_1$. 
Note that the subgraph of 
$G'$ induced by the vertex set $(C\setminus F)\cup \{u,v\}$ 
is a complete graph. Thus there exists a Hamilton path $P_2$ 
that connects
vertices $u,v$  and passes through all vertices in 
$(C\setminus F) \setminus S_2$. 
We conclude that 
$P_1\cup P_2$ is a Hamiltonian cycle in $G'\setminus S$, 
so $G'$ is $k$-Hamiltonian, this is a contradiction. 
\end{proof} 

If $s=k+1$, then $|F\cup V(H)|= |F| + |V(H)| = 
  k+1 +\delta -k=\delta +1$. 
Note that 
$G'[F, V(H)]$ forms a complete bipartite graph and 
$d_{G'}(v)=\delta$ for each $v\in V(H)$.  
This implies that $H$ is a complete subgraph on $\delta -k$ vertices and 
$F\cup V(H)$ is a clique on  $\delta +1$ vertices.  
In this case, we have $G'=L_{n,k,\delta}$ and then 
$G\subseteq L_{n,k,\delta}$.

If $s=\delta $, then by noticing 
 $G'[F, V(H)]$ forms a complete bipartite graph and 
$d_{G'}(v)=\delta$ for each $v\in V(H)$, 
we know that $V(H)$ is an independent set of order $\delta -k$. 
Thus we get $G'=H_{n,k,\delta}$ and then 
$G\subseteq H_{n,k,\delta}$. 
 The proof  is now complete.
\end{proof}

\noindent 
{\bf Remark.} 
We remark that the stability Theorem \ref{thm34} was partially proved in \cite{LLD2019} although the line of the proofs seems similar with that in \cite[Lemma 2]{LiBinlong}. 
However, the extremal graphs 
are characterized in \cite[Theorem 1.10]{LLD2019}  by using the terminology of graph-closure,
which states that the closure graph $\mathrm{cl}_{n+k} (G)\in \mathbb{G}_n(p,k+1,\delta)$,
where  $ \mathbb{G}_n(p,k+1,\delta)$  is a family of many graphs;
see \cite{LLD2019} for the  exact definition.
While in Theorem \ref{thm34} of the present paper,
we have shown that there are only two possible extremal graphs: $\mathrm{cl}_{n+k}(G)= H_{n,k,\delta}$
or $\mathrm{cl}_{n+k}(G)= L_{n,k,\delta}$, which implies
$G\subseteq H_{n,k,\delta}$ or $G\subseteq L_{n,k,\delta}$.
Thus, Theorem \ref{thm34}  could be viewed as a slight  improvement of \cite[Theorem 1.10]{LLD2019} in some sense.

\medskip 

As a direct consequence of Theorems \ref{thmFKL} and \ref{thm34}, 
we get the following corollary since we can verify that 
$e(H_{n,k,\delta +1})< e(H_{n,k,\delta})$ and $e(L_{n,k,\delta}) < 
e(H_{n,k,\delta})$ for $ n$ sufficiently large. 

\begin{corollary} 
Let $\delta >k \ge0 $ and $ n$ be sufficiently large. 
If $G$ is an $n$-vertex graph with 
minimum degree $\delta (G) \ge \delta$ 
and 
\[  e(G)\ge e(H_{n,k,\delta }), \]
then $G$ is $k$-edge-Hamiltonian and $k$-Hamiltonian unless 
$G= H_{n,k,\delta}$.  
\end{corollary}

We know that the proof strategy of Theorem \ref{thmniki} 
 presented in \cite{Niki16}
 does not apply  
 the stability theorem directly, it seems more algebraic than that in \cite{LiBinlong}. 
We remark here that there is another way to prove Theorem \ref{thmniki}. 
In fact, by a tiny modification of the proof of 
Theorem \ref{thmln16a} in \cite{LiBinlong}, 
we can prove that if $G$ is non-Hamiltonian and 
$\lambda (G)\ge n- \delta -1$, then 
$G$ is a  subgraph of $H_{n,\delta}$ or $L_{n,\delta}$. 
As pointed out by Nikiforov \cite{Niki16}, the crucial point of 
the argument of Theorem \ref{thmniki} 
is based on proving that for large $n\ge \delta^3+ \delta +4$, 
if $G$ is a subgraph of $H_{n,\delta}$ with $\delta (G)\ge \delta$, then 
$\lambda (G)< n- \delta -1$, unless $G=H_{n,\delta}$. 
The same argument holds for the subgraph of $L_{n,\delta}$. 
 We observe that both $H_{n,\delta}$ and $L_{n,\delta}$ consist 
 of a large clique $K_{n-\delta}$ together with 
 a few number of outgrowth edges, 
this implies that $\lambda (H_{n,\delta})$ and $\lambda (L_{n,\delta})$ 
 are slightly greater than $\lambda (K_{n- \delta})=n- \delta -1$. 
Roughly speaking, for sufficiently large $n$ 
 with respect to $\delta$,  
 the key idea of Nikiforov exploits the fact that when  
 $G$ is a subgraph of $H_{n, \delta}$ or $L_{n,\delta}$ with 
 $\delta (G)\ge \delta$, then 
 $G$ is obtained by deleting edges from the 
 subgraph $K_{n-\delta}$.
One can further show that all the outgrowth edges contribute to $\lambda (G)$ 
 much less than a single edge within $K_{n-\delta}$.  
 Thus  the deleting of any one edge from the clique 
 $K_{n-\delta}$ 
 can lead to 
 $\lambda (G)< \lambda (K_{n-\delta})=n-\delta-1$;  
see \cite[Theorem 1.6]{Niki16} for more details. 

\medskip 
With the above observation, one can prove the following 
analogues. 

\begin{theorem}[see \cite{LiuDMGT}] \label{thm35}
Let $k \ge 1, \delta \ge k+2$ and $n$ be sufficiently large. 
If $G$ is a subgraph of $H_{n,k,\delta}$ or $L_{n,k,\delta}$ 
and the minimum degree $\delta (G) \ge \delta$, then 
\[ \lambda (G) < n -\delta +k-1,  \]
unless $G=H_{n,k,\delta}$ or $G=L_{n,k,\delta}$. 
\end{theorem}

\noindent 
{\bf Remark.}  
When $G$ is a subgraph of $L_{n,k,\delta}$, 
this theorem was partially proved 
in  \cite[Theorem 1.6]{LiuDMGT}. 
The remaining case that $G$ is a subgraph of 
$H_{n,k,\delta}$ can be proved similarly, so we leave the details for  interested readers.

\section{Proofs of main results}

\label{sec4}

Recall that the signless Laplacian matrix $Q(G)$ 
 associated with graph $G$ is given as $D+A$, 
 where $D$ is the diagonal matrix of degrees 
 and $A$ is the adjacency matrix of $G$. 
 Let $q(G)$ denote the largest eigenvalue of $Q(G)$. 
 The Rayleigh principle yields 
 \[   q(G)= \max_{\bm{x} \neq \bm{0}} \frac{\bm{x}^TQ(G)\bm{x}}{\bm{x}^T\bm{x}},  \]
 where 
 \[  \bm{x}^TQ(G)\bm{x} 
 = \sum_{v\in V(G)} d(v) x_v^2 + 2 \sum_{\{u,v\}\in E(G)} x_ux_v .  \]
Let $\bm{f}$ be an eigenvector corresponding to 
$q(G)$, i.e., $Q(G)\bm{f}= q(G)\bm{f}$. 
By the celebrated Perron--Frobenius theorem (see \cite[p. 22]{BH2012} 
or \cite[p. 178]{GR2001}), 
we may assume that $f_{v}>0$ for each $v\in V(G)$ 
when $G$ is connected. 
It is easy to see from the eigen-equation that 
for any $u,v\in V(G)$,  
\begin{align*}
(q(G)- d(u)) f_u =\sum_{w\in N(u)} f_w, ~~\text{and}~~
(q(G)- d(v)) f_v =\sum_{z\in N(v)} f_z.
\end{align*}
Therefore, we obtain 
\begin{align}
\notag (q(G) - d(u))(f_u -f_v) &= 
(q(G) - d(u))f_u - (q(G)-d(v)) f_v  + (d(u) - d(v)) f_v \\
\notag &= (d(u) -d(v)) f_v + \sum_{s\in N(u)} f_s - \sum_{t\in N(v)} f_t \\
&=   (d(u) -d(v)) f_v + 
\sum_{s\in N(u)\setminus N(v)} f_s - 
\sum_{t\in N(v)\setminus N(u)} f_t.  \label{eqeq6}
\end{align}

In what follows, we shall spend a lot of efforts to prove some lemmas 
involving the signless Laplacian spectral radius. 
The main techniques used in this section applies the 
analytic method  on the entries of the Perron eigenvector. This technique 
originates from 
the recent work of Li, Liu and Peng \cite{LLPLMA18}. 

\begin{lemma} \label{lem31}
If $G\in \mathcal{H}_{n,k,\delta}^{(1)} $ or $G\in 
\mathcal{L}_{n,k,\delta}^{(1)}$, 
then $G$ is neither  $k$-edge-Hamiltonian  nor  $k$-Hamiltonian. Moreover, we have 
\[  q(G)\ge 2(n-\delta +k -1).  \] 
\end{lemma}

\begin{proof}
For each graph $G\in \mathcal{H}_{n,k,\delta}^{(1)} $ 
or $G\in 
\mathcal{L}_{n,k,\delta}^{(1)}$, 
we can clearly see that 
$G$ is not $k$-edge-Hamiltonian. 
Next we shall prove that  
$q(G)\ge 2(n-\delta +k -1)$. 
Recall the subsets 
$X,Y$ and $Z$ defined as above. 
For each case, we define a vector $\bm{h}$ such that 
$h_v =1$ for every $v\in Y\cup Z$ and 
$h_v=0$ for every $v\in X$. 
Note that $q(K_{n-\delta +k} \cup I_{\delta - k} ) =q(K_{n-\delta +k }) 
=2(n-\delta +k-1)$ 
and $\bm{h}$ is the corresponding eigenvector. 
If $G\in  \mathcal{H}_{n,k,\delta}^{(1)}$, then we get 
\[  \bm{h}^TQ(G)\bm{h} - \bm{h}^TQ ({K_{n-\delta +k} \cup I_{\delta - k}}) \bm{h} 
= \delta (\delta -k) - 4|E'| \ge 0. \]
By the Rayleigh Formula, we have 
\[ q(G) 
 \ge \frac{ \bm{h}^TQ (G)\bm{h} }{ \bm{h}^T\bm{h}  }  
 \ge \frac{\bm{h}^TQ ({K_{n-\delta +k} \cup I_{\delta -k}} )\bm{h}}{\bm{h}^T\bm{h}} 
 =2(n-\delta +k-1). 
\] 
Similarly, we can show that 
$q(G) \ge 2(n-\delta +k-1)$ for every $G\in  \mathcal{L}_{n,k,\delta}^{(1)}$. 
\end{proof}

We give the definitions of two families of graphs. 
\[ \mathcal{H}_{n,k,\delta}^{(2)} = 
\left\{ H_{n,k,\delta} \setminus E': E' \subseteq E_1(H_{n,k,\delta})
 ~\text{with}~|E'|\ge  
\lfloor {\delta (\delta -k)}/{4}\rfloor +1 \right\}, \]
and 
\[ \mathcal{L}_{n,k,\delta}^{(2)} = 
\left\{ L_{n,k,\delta} \setminus E': E' \subseteq E_1(L_{n,k,\delta})
 ~\text{with}~|E'|\ge 
\lfloor {(k+1) (\delta -k)}/{4}\rfloor +1 \right\}. \]

\begin{lemma} \label{lem32}
If $n$ is sufficiently large and 
$G\in \mathcal{H}_{n,k,\delta}^{(2)} $ or $G\in 
\mathcal{L}_{n,k,\delta}^{(2)}$, 
then  
\[  q(G)< 2(n-\delta +k -1).  \]
\end{lemma}

\begin{proof} 
In the following proof, we shall assume that 
$G\in \mathcal{H}_{n,k,\delta}^{(2)}$. 
Since the proof for the case of $G\in \mathcal{L}_{n,k,\delta}^{(2)}$ is  similar, we only give the sketch in this case. 
Let $G$ be a graph from $\mathcal{H}_{n,k,\delta}^{(2)}$ 
with maximum signless Laplacian spectral radius. 
This means that $G$ is obtained from $H_{n,k,\delta}$ 
by deleting exactly $\lfloor {\delta (\delta -k)}/{4}\rfloor +1$ 
edges from $E_1(H_{n,k,\delta})$ 
as the monotonicity of the signless 
Laplacian spectral radius. 
Let $\bm{f}$ be the eigenvector corresponding to $q(G)$. 
Furthermore, we assume that 
$\max_{v\in V(G)} f_v=1$.

Let $\bm{h}$ be the vector defined as in the proof 
of Lemma \ref{lem31}. 
First of all, we can show the following claim, 
which is a lower bound on $q(G)$.  

\medskip 
\noindent 
{\bf Claim 1.} \label{claim1} $q(G)>2(n-\delta +k -1) -1$.

\begin{proof}[Proof of Claim 1]
If $G\in  \mathcal{H}_{n,k,\delta}^{(2)}$, then we obtain 
\[  \bm{h}^TQ(G)\bm{h} - \bm{h}^TQ ({K_{n-\delta +k} \cup I_{\delta - k}}) \bm{h} 
= \delta (\delta -k) - 4|E'| \ge -4. \]
By the Rayleigh Formula, we have 
\[ q(G) 
 \ge \frac{ \bm{h}^TQ (G)\bm{h} }{ \bm{h}^T\bm{h}  }  
 \ge \frac{\bm{h}^TQ ({K_{n-\delta +k} \cup I_{\delta -k}} )\bm{h}}{\bm{h}^T\bm{h}} - \frac{4}{\bm{h}^T\bm{h}}
 =2(n-\delta +k-1) - \frac{4}{\bm{h}^T\bm{h}}, 
\] 
which implies that 
$q(G) \ge 2(n-\delta +k-1) -1$. 
\end{proof} 

Recall that 
\begin{align*}
X&= \{v\in V(H_{n,k,\delta}): d(v)=\delta\}, \\
Y&=\{v\in V(H_{n,k,\delta}): d(v)=n-1 \}, \\
Z&= \{v\in V(H_{n,k,\delta}): d(v)= n - \delta +k -1 \}. 
\end{align*}

We next show that all entries of $\bm{f}$ corresponding to 
the vertices of $X$ are tiny (or small) since 
$q(G)> 2n -2 \delta +2k -3$ by 
 Claim 1, and 
the condition that $n$ is large enough.  

\medskip 
\noindent 
{\bf Claim 2.} For each $x\in X$, we have 
\[  f_x \le \frac{\delta}{q(G)-\delta} = o(n). \]

\begin{proof}[Proof of Claim 2]
The following equality 
\[  (q(G) - d(x))  f_x = \sum_{y\in Y} f_y, \]
 together with $d(x)=\delta $, yields the required inequality. 
\end{proof}

Note that $E'$ is the edge set in which both 
endpoints are in $Y\cup Z$, 
we define two subsets of $Y$ and $Z$, respectively. 
\[  Y_1= \{y\in Y: d(y)=n-1\} ~\text{and}~ 
Y_2 = \{ y\in Y : d(y) \le n-2\}.  \]
Similarly, we define two subsets of $Z$ as follows. 
\[  Z_1=\{z\in Z : d(v)=n- \delta +k -1\} ~\text{and}~ 
Z_2= \{z\in Z: d(v)\le n - \delta +k -2\} . \]
Note that  the number of  edges removed from  $Y\cup Z$ 
is exactly $\lfloor \delta (\delta -k)/4\rfloor +1$.  
Since $|Z|=n-2\delta +k$ and $n$ is sufficiently large, 
we get  $Z_1\neq \varnothing$.   
We next compare the entries of eigenvector $\bm{f}$ 
corresponding to these subsets. 
Roughly speaking, the vertex with  large degree 
 always has large value at the corresponding 
 entry.

\medskip 
\noindent 
{\bf Claim 3.} We have the following statements. \\ 
(a) If $Y_2\neq \emptyset $,  
then $f_{y_2}<f_{z_1}$ 
for all $y_2\in Y_2$ and $z_1\in Z_1$.  \\ 
(b) If $Z_2\neq \emptyset $,  
then $f_{z_2}<f_{z_1}$ 
for all $z_2\in Z_2$ and $z_1\in Z_1$.  \\
(c) If $Y_1,Y_2\neq \emptyset $,  
then $f_{y_2}<f_{y_1}$ 
for all $y_2\in Y_2$ and $y_1\in Y_1$.  \\
(d) If $Y_1\neq \emptyset $,  
then $f_{z_1}<f_{y_1}$ 
for all $z_1\in Z_1$ and $y_1\in Y_1$. 

\begin{proof}[Proof of Claim 3]
(a) Suppose on the contrary that there exist some 
$y_2\in Y_2$ and $z_1\in Z_1$ such that 
$f_{y_2} \ge f_{z_1}$. 
  Let $w\in Z_2\cup Y_2$ be a vertex not adjacent to $y_2$.  
Since $z_1\in Z_1$, we have $\{z_1,y_2\}\in E(G)$ 
and $\{z_1,w\}\in E(G)$.  
We define a new graph $G^*\in \mathcal{H}_{n,k,\delta}^{(2)}$ by deleting $\{z_1,w\}$ and adding $\{y_2,w\}$. 
By Lemma \ref{lemhz}, 
we get $q(G^*)> q(G)$, which contradicts with 
the choice of $G$. Thus we have $f_{y_2}<f_{z_1}$ 
for all $y_2\in Y_2$ and $z_1\in Z_1$.


(b) For $z_2\in Z_2$ and $z_1\in Z_1$, 
we have
$\{z_1,z_2\}\in E(G)$.  
So $z_2\in N(z_1) \setminus N(z_2)$ and 
$z_1\in N(z_2) \setminus N(z_1)$. 
We denote $N[z_2]=N(z_2) \cup \{z_2\}$ and $N[z_1]=N(z_1) \cup \{z_1\}$. 
We can get from (\ref{eqeq6}) that 
\[ 
(q(G) - d(z_1) +1 )(f_{z_1} -f_{z_2}) 
=  (d(z_1) -d(z_2)) f_{z_2} + 
\sum_{s\in N(z_1)\setminus N[z_2]} f_s - 
\sum_{t\in N(z_2)\setminus N[z_1]} f_t. \] 
Since $N(z_2) \setminus \{z_1\} \subseteq N(z_1)\setminus \{z_2\}$, we then get 
\begin{equation} \label{eq+1}
 (q(G) - d(z_1) +1 )(f_{z_1} -f_{z_2}) 
=  (d(z_1) -d(z_2)) f_{z_2} + 
\sum_{s\in N(z_1)\setminus N[z_2]} f_s >0, 
\end{equation}
which implies $f_{z_1} >f_{z_2} $ since 
$q(G)\ge 2(n-\delta +k-1) -1 > d(z_1) -1$ and $d(z_1)>d(z_2)$. 
The proofs of (c) and (d) are similar 
with that of (a) and (b), so we omit the details. 
\end{proof}

\medskip 
\noindent 
{\bf Claim 4.} For each $w\in Y\cup Z$, we have 
\[ f_w \ge 1- 
\frac{(\delta +2)(\delta -k) +6}{2(q(G)-n+2)} = 1- o(n).\]   

\begin{proof}[Proof of Claim 4]
By Claim 3, it is sufficient to prove this claim 
in the case $w\in Y_2\cup Z_2$ whenever $Y_2\neq \emptyset$ or $Z_2 \neq \emptyset$. 
Since $Z_1$ is non-empty, by Claim 3 again, we know that 
$\max_{v\in V(G)} f_v$ is attained by vertices from 
$Y_1$ or $Z_1$. We next proceed in two cases. 

{\bf Case 1.} $Y_1= \emptyset$. 
By Claim 3, we get that $\max_{v\in V(G)} f_v$ is attained by vertices from  $Z_1$. 
Choose a vertex $z_1$  from $Z_1$ with $f_{z_1}=1$.  
We notice that $z_1$ is adjacent to all other vertices in 
$Y\cup Z$. 

{\bf Subcase 1.1.} If $w\in Z_2$, then $\{z_1,w\}\in E(G)$ and 
$N(w) \subseteq N(z_1)$, which yields 
$d(z_1)-d(w) \le \lfloor \delta (\delta -k) /4\rfloor +1$ and 
$|N(z_1) \setminus N[w]| \le  \lfloor \delta (\delta -k) /4\rfloor +1$. 
By applying (\ref{eq+1}), we have 
\[  (q(G) - d(z_1) +1 )(f_{z_1} -f_{w}) 
=  (d(z_1) -d(w)) f_{w} + 
\sum_{s\in N(z_1)\setminus N[w]} f_s \le 
\frac{\delta (\delta -k)}{2} +2 . \]
Note that $d(z_1)= n - \delta +k -1$ and $f_{z_1}=1$, 
we obtain 
\[  1- f_{w} \le \frac{\delta (\delta -k) +4}{2
(q(G) - n + \delta -k+2)}. \]

{\bf Subcase 1.2.} If $w\in Y_2$, then 
$ |d(z_1) - d(w)| \le  \frac{\delta (\delta -k)}{4}  + 1+
(\delta -k)$.  
Moreover, we have 
\[  |N(z_1)\setminus N[w]| \le \frac{\delta (\delta -k)}{4} +1,\quad  |N(w)\setminus N[z_1]|  = |X|=\delta -k.  \]
Applying (\ref{eq+1}), we get 
\begin{align*}
(q(G) - d(z_1) +1) (f_{z_1} -f_w) &= 
(d(z_1) - d(w)) f_w  +
\sum_{s\in N(z_1)\setminus N[w]} f_s - 
\sum_{t\in N(w)\setminus N[z_1]} f_t \\
& \le \frac{\delta (\delta -k)}{4} +1+ (\delta -k) + 
\frac{\delta (\delta -k)}{4} +1. 
\end{align*}
Note that $d(z_1)= n - \delta +k -1$ and $f_{z_1}=1$, 
we obtain 
\[ 1-f_w \le  \frac{(\delta+2)(\delta -k) +4}{2(q(G) - n + \delta -k+2)}.\]

{\bf Case 2.} $Y_1 \neq \emptyset$. 
By Claim 3, we get that $\max_{v\in V(G)} f_v$ is attained by vertices from  $Y_1$. 
Let $y_1$ be a vertex from $Y_1$ such that $f_{y_1}=1$. 
By repeating the argument in Case 1, we can prove that 
\[  1 - f_w \le \frac{(\delta +2)(\delta -k) +6}{2(q(G)-n+2)} . \] 
The proof  of Claim 4 is completed. 
\end{proof}

Claims 2 and 4 showed that 
the Perron eigenvector $\bm{f}$ has small values on 
the entries of $X$ and large values on the entries 
of $Y\cup Z$. This reveals that  vector $\bm{f}$ is entrywise  close to  $\bm{h}$  defined as in the proof of 
Lemma \ref{lem31}.

We  are now ready to prove Lemma \ref{lem32} completely. 
By Claims 2 and 4, we obtain  
\begin{align*}
& \bm{f}^TQ(G)\bm{f} - \bm{f}^TQ ({K_{n-\delta +k} \cup I_{\delta - k}}) \bm{f} \\ 
& = \sum_{x\in X,y\in Y} (f_x + f_y)^2 - 
\sum_{\{u,v\}\in E'} (f_u + f_v)^2 \\
& \le \delta (\delta -k) \left( \frac{\delta}{ q(G)- \delta} +1 \right)^2 
-4 |E'| \left( 1- 
\frac{(\delta -2)(\delta -k) +6}{2(q(G)- n +2)}  \right)^2 \\
&<0,
\end{align*} 
where the last inequality holds 
by using $q(G)>2(n-\delta +k -1) -1$ in Claim 1  
and $|E'|=\lfloor \frac{\delta (\delta -k )}{4} \rfloor +1 > 
\frac{\delta (\delta -k)}{4}$, 
then $\frac{\delta (\delta -k)}{4|E'|} < \frac{1- o(1)}{1+ o(1)}$ holds for
sufficiently large  $n$. 
Hence, we have 
\[ q(G)=\frac{\bm{f}^TQ(G)\bm{f}}{\bm{f}^T\bm{f}}
 < \frac{\bm{f}^TQ ({K_{n-\delta +k} \cup I_{\delta - k}}) \bm{f}
 }{\bm{f}^T\bm{f}} \le 2(n-\delta +k -1). \]
This completes the proof. 
\end{proof}

With the help of these Lemmas, 
we can prove our theorems immediately.

\begin{proof}[{\bf Proof of Theorems \ref{thm21} and \ref{thm25}}]
By Theorem \ref{thm31}, we can get 
\[  n- \delta + k-1 \le 
\lambda (G) \le \frac{1}{2}\Bigl( \delta-1 + 
\sqrt{8e(G) - 4\delta n + (\delta+1)^2} \Bigr).   \]
By calculation, and noticing that $n$ is large enough, 
 we have 
\begin{align*} 
e(G)&\ge  \frac{n^2 - (2\delta -2k +1)n 
+ 2\delta^2 -3\delta k + \delta +k^2-k}{2} \\
& > \tbinom{n -\delta -1 +k }{ 2} + 
(\delta +1)(\delta +1-k) \\
&= e(H_{n,k,\delta +1}). 
\end{align*}
By Theorems \ref{thmFKL} and \ref{thm34}, we know that 
$G$ is $k$-edge-Hamiltonian and 
$k$-Hamiltonian unless 
$G\subseteq H_{n,k,\delta}$ or $G\subseteq L_{n,k,\delta}$.  Note that $\delta (G)\ge \delta$ 
and $\lambda (G) \ge n- \delta + k-1$. 
By Theorem \ref{thm35},  we get 
$G=H_{n,k,\delta}$ or $G=L_{n,k,\delta}$. 
\end{proof}

\begin{proof}[{\bf Proof of Theorems \ref{thm23} and 
\ref{thm27}}]
By applying Theorem \ref{thmFY}, we obtain 
\[ 2(n- \delta +k -1) \le q(G) \le \frac{2e(G)}{n-1} + n-2. \] 
Since $n$  is sufficiently large, we  have 
\[  e(G) \ge \frac{n^2- (2\delta -2k +1)n +2\delta -2k}{2} 
> e(H_{n,k,\delta +1}).  \]
By Theorems \ref{thmFKL} and \ref{thm34}, we get that 
$G$ is $k$-edge-Hamiltonian and $k$-Hamiltonian unless 
$G\subseteq H_{n,k,\delta}$ or $G\subseteq L_{n,k,\delta}$. Note that $\delta (G)\ge \delta$ and $q(G) \ge 2(n-\delta +k-1)$. 
By Lemma  \ref{lem32}, 
we know that $G\in \mathcal{H}_{n,k,\delta}^{(1)} $ or $G\in 
\mathcal{L}_{n,k,\delta}^{(1)}$. 
\end{proof}

\section{Concluding remarks}

\label{sec5}

Let $G$ be a bipartite graph with vertex sets $X$ and $Y$. 
The bipartite graph $G$ is called balanced if $|X|=|Y|$. 
If  a bipartite graph has Hamilton cycle, 
then it must be balanced. 
So we  consider the existence of Hamilton cycle only 
in balanced bipartite graphs.

Motivated by the work of Erd\H{o}s \cite{erdos62} 
in Theorem \ref{thmerdos62}, 
Moon and Moser \cite{MM63} provided a corresponding result 
for the balanced bipartite graphs.

\begin{theorem}[Moon--Moser \cite{MM63}] \label{thmmm}
Let $G$ be a balanced bipartite graph on $2n$ vertices.  
If the minimum degree $\delta (G)\ge \delta $ for some $1\le \delta \le n/2$ and 
\[  e(G) >  n(n-\delta) + \delta^2, \]
then $G$ has a Hamilton cycle. 
\end{theorem} 

The condition $ \delta \le {n}/{2}$ is well natural  
since Moon and Moser \cite{MM63} also pointed out that 
if $G$ is a balanced bipartite on $2n$ vertices with $\delta (G) > {n}/{2}$, 
then $G$ must be Hamiltonian. 
This is a bipartite version of the Dirac theorem.

Let  $B_{n,\delta}$ be the bipartite graph  
obtained from the complete bipartite graph $K_{n,n}$ 
by deleting all edges in its one subgraph $K_{\delta,n-\delta}$. 
More precisely, the two vertex parts of $B_{n,\delta}$ 
are $V=V_1\cup V_2$ and $U=U_1 \cup U_2$ 
where $|V_1|=|U_1|=\delta$ and $|V_2|=|U_2|=n-\delta$, 
we join all edges between $V_1$ and $U_1$, and all edges between 
$V_2$ and $U$. It is easy to see that 
$e(B_{n,\delta})= n(n-\delta) + \delta^2$ and $B_{n,\delta}$ 
contains no Hamilton cycle. This implies that 
 the condition in Moon--Moser's theorem is best possible.  

For the Hamiltonicity of balanced bipartite graphs, 
Li and Ning \cite[Theorem 1.10]{LiBinlong} also proved the spectral version of 
the Moon--Moser Theorem \ref{thmmm}, that is, 
the bipartite version of Theorems \ref{thmln16a} and \ref{thmln16b}. 
On the other hand, a bipartite graph is called nearly balanced if $\bigl| |X|-|Y| \bigr| \le 1$. 
Additionally, Li and Ning \cite{LNLAA17} also considered 
the existence of Hamilton path in nearly balanced bipartite graphs. 
Later, Ge and Ning \cite[Theorem 1.4]{Ningbo20}, 
and Jiang, Yu and Fang \cite[Theorem 1.2]{JYF2019} independently 
proved a further improvement on the 
adjacency spectral result of Li and Ning for  
balanced bipartite graphs. 
Correspondingly, 
Li, Liu and Peng \cite[Theorem 4]{LLPLMA18} 
also gave a further improvement 
on the signless Laplacian spectral result of Li and Ning.  
It is noteworthy that Liu, Wu and Lai \cite{LWL2020} 
unified  several former spectral Hamiltonian results 
on balanced bipartite graphs.  
In addition, Lu \cite{Lu2020} 
extended some spectral conditions for the Hamiltonicity
of balanced bipartite graphs. 

\begin{question}
 It is possible and interesting to extend the notations of 
 $k$-edge-Hamiltonian and $k$-Hamiltonian,  
and establish the variants of our results in 
Section \ref{sec2} for the balanced bipartite graphs or nearly balanced bipartite graphs. 
\end{question}

\section*{Acknowledgements}

This work was  supported by
 NSFC (Grant Nos. 11931002, 11671124).  
We would like to thank the anonymous referees for their 
careful reviews.

\end{document}